\documentclass[11pt,reqno]{amsart}
\usepackage[T1]{fontenc}
\usepackage[utf8]{inputenc}
\usepackage[a4paper,margin=3cm]{geometry}
\usepackage[pagebackref]{hyperref}
\usepackage{latexsym,lmodern}

\title[Second order universality at edge of Coulomb gases on the plane]{A note on the second order universality at~the~edge~of~Coulomb~gases~on~the~plane}

\author{Djalil~Chafaï} %
\address{Djalil Chafaï, Université Paris Dauphine, CEREMADE, IUF, PSL, France} %
\author{Sandrine~Péché}
\address{Sandrine~Péché, Université Paris-Diderot, LPMA, France}
\date{October 2013. Revised April 2014. Accepted in Journal of Statistical Physics.}

\numberwithin{equation}{section}



\theoremstyle{plain}
\newtheorem{theorem}{Theorem}[section]%
\newtheorem{lemma}[theorem]{Lemma}%

\newtheorem{remark}[theorem]{Remark}%
\newcommand{\dC}{\mathbb{C}}
\newcommand{\dE}{\mathbb{E}}

\newcommand{\dP}{\mathbb{P}}
\newcommand{\dR}{\mathbb{R}}

\newcommand{\cC}{\mathcal{C}}

\newcommand{\cM}{\mathcal{M}}\newcommand{\cN}{\mathcal{N}}
\newcommand{\cO}{\mathcal{O}}

\newcommand{\al}{\alpha}
\newcommand{\be}{\beta}

\newcommand{\de}{\delta}

\newcommand{\la}{\lambda}

\newcommand{\si}{\sigma}

\newcommand{\te}{\theta}

\newcommand{\R}{\mathbb{R}}
\newcommand{\veps}{\varepsilon}

\newcommand{\ABS}[1]{{{\left| #1 \right|}}} % |1|
\newcommand{\BRA}[1]{{{\left\{#1\right\}}}} % {1}
\newcommand{\NRM}[1]{{{\left\| #1\right\|}}} % ||1||
\newcommand{\PAR}[1]{{{\left(#1\right)}}} % (1)
\newcommand{\SBRA}[1]{{{\left[#1\right]}}} % [1]
\newcommand{\IND}{\mathbf{1}}

\newcommand{\OL}[1]{\overline{#1}}

\makeatletter
\def\@MRExtract#1 #2!{#1}     % thanks, Martin!
\renewcommand{\MR}[1]{% we need to strip the "(...)"
  \xdef\@MRSTRIP{\@MRExtract#1 !}%
  \href{http://www.ams.org/mathscinet-getitem?mr=\@MRSTRIP}{MR-\@MRSTRIP}}
\makeatother
\begin{document}

\begin{abstract}
  We consider in this note a class of two-dimensional determinantal Coulomb
  gases confined by a radial external field. As the number of particles tends
  to infinity, their empirical distribution tends to a probability measure
  supported in a centered ring of the complex plane. A quadratic confinement
  corresponds to the complex Ginibre Ensemble. In this case, it is also
  already known that the asymptotic fluctuation of the radial edge follows a
  Gumbel law. We establish in this note the universality of this edge
  behavior, beyond the quadratic case. The approach, inspired by earlier works
  of Kostlan and Rider, boils down to identities in law and to an instance of
  the Laplace method.
\end{abstract}

\maketitle

% {\footnotesize\tableofcontents}

\section{Introduction}

Let us consider a gas of charged particles $\{z_1,\ldots,z_n\}$ on the complex
plane $\dC$, confined individually by the external field $Q$ and experiencing
a Coulomb pair repulsive interaction. This corresponds to the probability
distribution on $\dC^n$ with density proportional to
\begin{equation}\label{eq:cougas}
  (z_1,\ldots,z_n)\in\dC^n %
  \mapsto %
  \prod_{j=1}^ne^{-nQ(z_j)}\prod_{1\leq j<k\leq{}n}\ABS{z_j-z_k}^\be,
\end{equation}
where $\be>0$ is a fixed parameter and where $Q:\dC\to\dR$ is a fixed smooth
function. We are mostly interested in asymptotics in $n\to\infty$ of this
particles system. To ensure the integrability for large enough $n$, and
following \cite{ECP1818}, it is convenient to assume that there exists some
real number $\be'>\be$ with $\be'\geq1$ such that
$c:=\sup_{z\in\dC}\{(1+\ABS{z})^{\be'}e^{-Q(z)}\}<\infty$. Indeed, using the
inequality $\ABS{a-b}\leq(1+\ABS{a})(1+\ABS{b})$ valid for any $a,b\in\dC$, we
get
\[
\prod_{j=1}^ne^{-nQ(z_j)}\prod_{j<k}\ABS{z_j-z_k}^{\be}
\leq \prod_{j=1}^n\PAR{\PAR{1+\ABS{z_j}}^\be e^{-Q(z_j)}}^n
\leq c^{n^2}\prod_{j=1}^n\PAR{1+\ABS{z_j}}^{-n(\be'-\be)}.
\]
The factor $n$ in front of $Q$ in the density \eqref{eq:cougas} is motivated
by the remarkable formula
\[
\sum_{j=1}^nnQ(z_j)-\be\sum_{j<k}\log\ABS{z_j-z_k}
=n^2\PAR{\int\!Q(z)\,d\mu_n(z)
+\frac{\be}{2}\iint_{z\neq w}\!\!\!\!\log\frac{1}{\ABS{z-w}}\,d\mu_n(z)d\mu_n(w)}.
\]
where $\mu_n:=\frac{1}{n}\sum_{k=1}^n\de_{z_k}$ is the empirical distribution
of the particles. Indeed, following
\cite{MR1660943,MR1606719,MR2817648,ECP1818} (see \cite{cgz} for more general
models), on the space $\cM_1(\dC)$ of probability measures on $\dC$ equipped
with the topology of narrow convergence (i.e.\ the dual topology related to
bounded continuous test functions), the functional
\[
\mu\in\cM_1(\dC)\mapsto
I_Q(\mu)=\frac{1}{2}
\iint_{\dC^2}\!\PAR{Q(z)+Q(w)+\be\log\frac{1}{|z-w|}}\,d\mu(z)d\mu(w)
\]
is strictly convex, bounded from below with compact level sets, admits a
unique minimizer $\mu_Q$, and the empirical distribution
$\mu_n=\frac{1}{n}\sum_{k=1}^n\de_{z_k}$ satisfies a large deviations
principle for the weak topology at speed $n^2$ with good rate function
$I_Q-I_Q(\mu_Q)$. In particular, from the first Borel-Cantelli lemma, with
probability one, $\mu_n\to\mu_Q$ weakly as $n\to\infty$. Following \cite[Th.\
1.3]{MR1485778}, it can be shown that $\mu_Q$ has compact support when
$\lim_{\ABS{z}\to\infty}|z|^\be e^{-Q(z)}=0$.

\begin{remark}[Random matrices]
  The Coulomb gas \eqref{eq:cougas} is the spectrum of an Ensemble of random
  normal matrices obtained by conjugating with an independent Haar unitary
  matrix, see for instance \cite{MR1643533,MR2172690,MR2817648}. On the other
  hand, such a Coulomb gas also describes the spectrum of some few non normal
  random matrix ensembles. For instance, if $A$ is a random $n\times n$ matrix
  with i.i.d.\ complex Gaussian entries $\cN(0,\frac{1}{2n}I_2)$ of mean $0$
  and variance $1/n$ then its eigenvalues are a Coulomb gas \eqref{eq:cougas}
  with $\beta=2$ and $Q(z)=\ABS{z}^2$, and in this case $\mu_Q$ is the uniform
  distribution on the unit disc of $\dC$, see
  \cite{MR0173726,khoruzhenko-sommers,MR2908617}. Furthermore, if $B$ is an
  independent copy of $A$, then the spectrum of $AB^{-1}$ is a Coulomb gas
  \eqref{eq:cougas} with $\be=2$ and $Q(z)=\log(1+\ABS{z}^2)$, and in this
  case, $\mu_Q$ is not compactly supported and is actually heavy tailed, see
  \cite{MR2489167,ECP1818}. In both examples, $\be=2$ and $Q$ is radially
  symmetric.
\end{remark}

We suppose from now on that $Q$ is radially symmetric in the sense that
\[
Q(z)=V(|z|)
\]
where $V:\dR_+\to\dR$ is smooth and such that either $r\mapsto rV'(r)$ is
increasing or $V$ is convex, with $r^\be e^{-V(r)}\to0$ as $r\to\infty$. Let
$r_0$ be the smallest number such that $V'(r)>0$ for all $r>r_0$, and let
$R_0$ be the smallest solution of $R_0V'(R_0)=\be$. Then $0\leq
r_0<R_0<\infty$ and, following \cite[Th.\ 6.1]{MR1485778}, the probability
measure $\mu_Q$ is, in polar coordinates ($z=re^{i\te}$),
\[
d\mu_Q(z)=\frac{1}{2\pi\be}(rV'(r))'\IND_{[r_0,R_0]}(r)\IND_{[0,2\pi]}(\te)\,drd\te.
\]
This formula comes in fact from $dxdy=rdrd\te$ and $r\Delta
Q(z)=rV''(r)+V'(r)=(rV'(r))'$. Let $\{z_1,\ldots,z_n\}$ be the Coulomb gas
\eqref{eq:cougas} (in other words the atoms of $\mu_n$) and
\begin{equation}\label{eq:mods}
  \ABS{z}_{(1)}\geq\cdots\geq\ABS{z}_{(n)}
\end{equation}
be their reordered moduli (in other words the order statistics of the moduli).
In particular,
\[
\ABS{z}_{(1)}=\max_{1\leq k\leq n}\ABS{z_k} %
\quad\text{and}\quad %
\ABS{z}_{(n)}=\min_{1\leq k\leq n}\ABS{z_k}.
\]
Following Rider \cite{MR1986426} (see also \cite{rider-sinclair}), in the case
$\be=2$ and $V(r)=r^2$ then $\ABS{z}_{(1)}\overset{\dP}{\to}1$ as $n\to\infty$
and the asymptotic fluctuation follows a Gumbel law. The aim of this note is
to show, still for $\be=2$, the universality of this result beyond the
quadratic case on $V$. The following structural result forms the basement of
all the remaining results of this note.

\begin{theorem}[Layered structure]\label{th:layers}
  If $\beta=2$ then we have the identity in distribution
  \[
  (\ABS{z}_{(1)},\ldots,\ABS{z}_{(n)})
  \overset{d}{=}
  (R_{(1)},\ldots,R_{(n)})
  \]
  where $\ABS{z}_{(1)}\geq\cdots\geq\ABS{z}_{(n)}$ are as in \eqref{eq:mods}
  and where $R_{(1)}\geq\cdots\geq R_{(n)}$ is the order statistics of
  independent random variables $R_1,\ldots,R_n$ with $R_k$ of density
  proportional to 
  \[
  t\mapsto t^{2k-1}e^{-nV(t)}\mathbf{1}_{[0,\infty)}(t),
  \]
  for every $1\leq k\leq n$. In other words, in terms of point processes this
  means
  \[
  \sum_{k=1}^n\de_{\ABS{z_k}}
  \overset{d}{=}
  \sum_{k=1}^n\de_{R_k}.
  \]
\end{theorem}

Theorem \ref{th:layers} follows from the approach of Kostlan \cite{MR1148410}
used in the Gaussian case $V(r)=r^2$ (see also \cite[Theorem 4.7.1]{MR2552864}
for a determinantal point processes point of view). Let us quickly recall the
proof for convenience, which is elementary.

\begin{proof}[Proof of Theorem \ref{th:layers}]
  One starts from the Vandermonde determinant
  \[
  \prod_{1\leq j<k\leq n}|z_j-z_k| %
  =\det\PAR{\PAR{z_j^{k-1}}_{1\leq j,k\leq n}}
  =\ABS{\sum_{\si\in S_n}(-1)^{\mathrm{sig}(\si)}\prod_{j=1}^n z_j^{\si(j)-1}}.
  \]
  Hence, the density \ref{eq:cougas} rewritten in polar coordinates
  ($z_j=r_je^{i\te_j}$) is proportional to
  \[
  (r_1,\ldots,r_n,\te_1,\ldots,\te_n)\mapsto
  e^{-n\sum_{j=1}^nV(r_j)}\ABS{\sum_{\si\in S_n}(-1)^{\mathrm{sig}(\si)}\prod_{j=1}^nr_j^{\si(j)-1}e^{i(\si(j)-1)\te_j}}^2
  \prod_{j=1}^n r_j.
  \]
  At this step, we observe that if $\si,\si'\in S_n$ with $\si(k)\neq\si'(k)$ for some $k$ then
  \[
  \int_0^{2\pi}\!
  \PAR{\prod_{j=1}^n%r_j^{\si(j)-1}
    e^{i(\si(j)-1)\te_j}}
  \OL{\PAR{\prod_{j=1}^n%r_j^{\si'(j)-1}
      e^{i(\si'(j)-1)\te_j}}}\,d\te_k
  =0.
  \]
  Consequently, we obtain
  \[
  \int_{[0,2\pi]^n}\!
  \prod_{1\leq j<k\leq n}|z_j-z_k|^2\,d\te_1\cdots d\te_n
  =(2\pi)^n\sum_{\si\in S_n}\prod_{j=1}^nr_j^{2(\si(j)-1)}
  =(2\pi)^n\mathrm{per}\SBRA{r_j^{2(k-1)}}_{1\leq j,k\leq n}
  \]
  (``$\mathrm{per}$'' stands for ``permanent'') and thus, the density of the
  moduli is proportional to
  \[
  e^{-n\sum_{j=1}^n V(r_j)}%
  \mathrm{per}\SBRA{r_j^{2(k-1)}}_{1\leq j,k\leq n}\prod_{j=1}^nr_j %
  = \mathrm{per}\SBRA{r_j^{2k-1}e^{-nV(r_j)}}_{1\leq j,k\leq n}
  = \mathrm{per}\SBRA{f_k(r_j)}_{1\leq j,k\leq n}.
  \]
  Up to a factor $1/n!$, this is the density of $R_{\si(1)},\ldots,R_{\si(n)}$
  where $R_1,\ldots,R_n$ are independent random variables with $R_k$ of
  density proportional to $t\mapsto
  t^{2k-1}e^{-nV(t)}\mathbf{1}_{[0,\infty)}(t)$ for every $1\leq k\leq n$, and
  where $\si$ is a random uniform element of $S_n$, independent of
  $R_1,\ldots,R_n$.
\end{proof}

Beware that the phases are not independent in Theorem \ref{th:layers}.
However, since the gas is rotationally invariant, we may deduce that the
equilibrium measure $\mu_Q$ exists, is rotationally invariant, and we may
compute its radial part using the law of the moduli.

If $V(r)=r^2$ then $\sqrt{n}R_k$ is $\chi$-distributed with $2k$
degree of freedom. Theorem \ref{th:layers} allows to reuse the
strategy behind the work of Rider \cite{MR1986426} in the case where
$V$ is a power. The following theorem is proved in Section
\ref{se:power}. The case $\al=2$ is the one of Rider \cite{MR1986426}.

\begin{theorem}[Power case]\label{th:power}
  Let $\ABS{z}_{(1)}=\max_{1\leq k\leq n}\ABS{z_k}$ be as in
  \eqref{eq:mods}, with $\be=2$. Suppose that $V(t)=t^\al$ for any
  $t\geq0$, for some $\al\geq1$. Set $c_n:=
  \log(n)-2\log\log(n)-\log(2\pi)$ and
  \[
  a_n:=2\PAR{\frac{\al}{2}}^{1/\al+1/2}\sqrt{nc_n} %
  \quad\text{and}\quad %
  b_n:=
  \PAR{\frac{2}{\al}}^{1/\al}
  \PAR{1+\frac{1}{2}\sqrt{\frac{2}{\al}\frac{c_n}{n}}}.
  \]
  % In accordance with Rider, for $\al=2$ we get
  % \[
  % a_n:= 2\sqrt{nc_n} %
  % \quad\text{and}\quad %
  % b_n:=\sqrt{1+\frac{\sqrt{c_n}}{\sqrt n}} %
  % =1+\frac{\sqrt{c_n}}{2\sqrt{n}}+\cO(c_n/n).
  % \]
  Then $(a_n(\ABS{z}_{(1)}-b_n))_{n\geq1}$ converges in distribution as
  $n\to\infty$ to the standard Gumbel law:
  \[
  \forall x\in\dR,\quad
  \lim_{n\to\infty}\dP(a_n(\ABS{z}_{(1)}-b_n)\leq x)
  =e^{-e^{-x}}.
  \]
  In particular ${(\ABS{z}_{(1)})}_{n\geq1}$ converges in probability to
  $(2/\al)^{1/\al}$ as $n\to\infty$, in other words
  \[
  \forall\veps>0,\quad
  \lim_{n\to\infty}\dP\PAR{\ABS{\ABS{z}_{(1)}-\PAR{2/\al}^{1/\al}}>\veps}=0.
  \]
\end{theorem}

%As we have already mentioned, in the quadratic case $\al=2$, the
%Coulomb gas \eqref{eq:cougas} describes the eigenvalues distribution
%of the so called complex Ginibre Ensemble of random matrices with
%i.i.d.\ complex Gaussian entries of law $\cN(0,\frac{1}{2n}I_2)$. 
Our next result below, which is proved in Section \ref{se:genecase}, goes
beyond the power case considered in Theorem \ref{th:power} (which corresponds
formally to the special choice $V(t)=t^\al$).

\begin{theorem}[Beyond the power case]\label{th:genecase}
  Let $\ABS{z}_{(1)}=\max_{1\leq k\leq n}\ABS{z_k}$ be as in \eqref{eq:mods},
  with $\be=2$. Additionally, let us assume the following properties on the
  potential $V$:
  \begin{itemize}
  \item \textbf{(A1)} $V$ is strictly convex: there exists $a>0$ such that
    $V''(u)\geq a, \forall u\geq 0;$
  \item \textbf{(A2)} For each $x\in [0,2]$, there exists a unique $t_x$ such
    that $t_x V'(t_x)=2-x.$
  \end{itemize}
  Let $t_0$ be the unique solution to the equation : $t_0 V'(t_0)=2$.
  Let us define the sequences
  \[
  a_n:=\frac{\sqrt{nc_n}}{C_0}
  \quad\text{and}\quad
  b_n:=t_0+C_0\sqrt{\frac{c_n}{n}}
  \]
  where
  $C_0:=1/\sqrt{\ABS{2/t_0^2+V''(t_0)}^{3/2}t_0/2}$ and
  \[
  c_n:=2\log\PAR{\frac{\sqrt{n/(2\pi)}}{\log(n)}}=\log(n)-2\log\log(n)-\log(2\pi).
  \]
  Then ${(a_n(\ABS{z}_{(1)}-b_n))}_{n\geq1}$ converges in distribution
  as $n\to\infty$ to the standard Gumbel law:
  \[
  \forall x\in\dR,\quad
  \lim_{n\to\infty}\dP(a_n(\ABS{z}_{(1)}-b_n)\leq x)=e^{-e^{-x}}.
  \]
  In particular, ${(\ABS{z}_{(1)})}_{n\geq1}$ converges in probability
  to $t_0$ as $n\to\infty$, in other words,
  \[
  \forall\veps>0, \quad
  \lim_{n\to\infty}\dP\PAR{\ABS{\ABS{z}_{(1)}-t_0}>\veps}=0.
  \]
\end{theorem}

\subsection*{Weakly confining potentials and heavy tails}

If the confining potential is not strong enough, the asymptotic fluctuation of
$\ABS{z}_{(1)}$ is no longer Gumbel. For instance, in the case where
$V(t)=c\log(1+t^2)$ for $c>1$, we get from Theorem \ref{th:layers} that $R_k$
has density proportional to $t\mapsto
t^{2k-1}/(1+t^2)^{cn}\mathbf{1}_{[0,\infty)}(t)$. In this case, the muduli of
the particles are heavy tailed, and the equilibrium measure $\mu_Q$ exists and
is heavy tailed. One may wonder if the largest particle in modulus
$\ABS{z}_{(1)}$ has Fr\'echet type fluctuations as $n\to\infty$. Recall that a
random variable $H$ follows the Fr\'echet law of parameter $\al>0$ iif
$\dP(H\leq t)=\exp(-t^{-\al})$ for any $t>0$. In particular, $\dP(H\leq
t)\approx 1-t^{-\al}$ as $t\gg1$. At least formally, if one takes
\[
V(t)=V_n(t)=+\infty\mathbf{1}_{0<t<1}+(2(1+1/n))\log(t)\mathbf{1}_{t\geq1}
\]
then $t^{2k-1}e^{-nV(t)}=t^{-2(n+1-k)-1}\mathbf{1}_{t\geq1}$ and thus
for every $t\geq1$
\[
\mathbb{P}(\max_{1\leq k\leq n}R_k\leq t) %
=\prod_{k=1}^n\mathbb{P}(R_k\leq t) %
=\prod_{k=1}^n(1-t^{-2(n+1-k)}) %
=\prod_{k=1}^{n}(1-t^{-2k})
\]
which gives
\[
\log\mathbb{P}(\max_{1\leq k\leq n}R_k\leq t)
=\sum_{k=1}^n\log(1-t^{-2k})\overset{t\gg1}{\approx} %
-\sum_{k=1}^nt^{-2k} \overset{n\gg1}{\approx}-1/(t^2-1)
\]
and therefore 
\[
\mathbb{P}(\max_{1\leq k\leq n}R_k\leq t) %
\overset{t,n\gg1}{\approx}e^{-1/(t^2-1)}
\approx 1-1/(t^2-1)
\approx 1-1/t^2=\dP(R_n\leq t).
\]
The phenomenon is that a sum behaves asymptically like its largest term, which
implies for the particles that the largest in modulus behaves like the one of
largest index.

\subsection*{Comments and open problems}

The edge universality remains untouched for general $\beta$-ensembles on the
complex plane and for the eigenvalues of general random matrices with i.i.d.\
entries (spectral radius). The complex Ginibre ensemble is exactly solvable
and belongs to both categories. The method used in the proof of Theorem
\ref{th:layers} is maybe still usable when $\beta$ is an even integer. One may
also use it for gases in $\dR^d$, $d>3$, with density proportional to
$\prod_{j=1}^ne^{-nQ(z_j)}\prod_{1\leq j<k\leq n}\NRM{z_j-z_k}_2^2$, in order
to obtain the law of the norms of the particles (by integrating non radial
variables, possibly via a sort of Wick formula on spheres). A study of the
bulk universality beyond logarithmic repulsion is considered in
\cite{2012arXiv1205.0671G}. On the other side, it is conjectured that the
spectral radius of square random matrices with i.i.d.\ centered entries of
variance $1/n$ and finite fourth moment has also a Gumbel type asymptotic
fluctuations as $n\to\infty$, making the Ginibre case universal again.

One may ask about the universality of refined aspects of the complex Ginibre
ensemble, such as the order statistics of the moduli of the eigenvalues,
studied by Rider in \cite{MR2035641}. One may ask about the fluctuation of the
smallest particle in modulus $\ABS{z}_{(n)}$ instead on the largest
$\ABS{z}_{(1)}$. For a general radially symmetric $Q$, the equilibrium measure
$\mu_Q$ is radially symmetric but might be supported by more than one ring
(lack of radial connectivity). One may then study the fluctuation at each edge
(inner and outer) of these rings. Another natural question is to ask about the
fluctuation at the edge in the single ring theorem \cite{MR3000558}.

\subsection*{Acknowledgments} 

This note benefited from the comments of two anonymous reviewers.

\section{Proof of Theorem \ref{th:power}}
\label{se:power}

\begin{proof}[Proof of Theorem \ref{th:power}]
  From Theorem \ref{th:layers}, $n\ABS{z}_{(k)}^\al$ has density proportional
  to $t^{\frac{2k}{\al}-1}e^{-t}$ which is $\Gamma(2k/\al,1)$, for every
  $1\leq k\leq n$. This gives the identity in distribution
  \[
  n\ABS{z}_{(k)}^\al\overset{d}{=}X_{1}+\cdots+X_k
  \]
  where $X_1,\ldots,X_k$ are i.i.d.\ of law $\Gamma(2/\al,1)$ (mean and
  variance both equal to $2/\al$). Thus, for any deterministic sequences
  $(A_n)$ and $(B_n)$ in $(0,\infty)$, and every $x\in\dR$,
  \begin{equation*}\label{eq:idd}
    \dP\PAR{A_n(\ABS{z}_{(1)}^\al-B_n)\leq x}
    =\prod_{k=1}^n\dP%
    \PAR{\frac{X_1+\cdots+X_k}{n}\leq \frac{x}{A_n}+B_n}.
  \end{equation*}
  Let $c_n:= 2\log\PAR{\frac{\sqrt{n/(2\pi)}}{\log(n)}}=\log(n)-2\log\log(n)-\log(2\pi).$ Lemma
  \ref{le:rider} gives now that for 
  \[
  A_n=\sqrt{\frac{\al}{2}nc_n}
  \quad\text{and}\quad
  B_n=\frac{2}{\alpha}+\sqrt{\frac{2}{\alpha}\frac{c_n}{n}}
  \]
  and for every $x\in\dR$,
  \begin{equation*}\label{eq:Pb}
    \lim_{n\to\infty}\prod_{k=1}^n\dP%
    \PAR{\frac{X_1+\cdots+X_k}{n}\leq \frac{x}{A_n}+B_n}%
    = e^{-e^{-x}}.
  \end{equation*}
  It remains finally to use Lemma \ref{le:delta} with $f(x)=x^{1/\al}$ and
  $Z_n=\ABS{z}_{(1)}=\max_{1\leq k\leq n}\ABS{z_k}$ (one has also to use the
  Slutsky lemma to obtain a bit nicer shift parameter $b_n$).
\end{proof}

\begin{lemma}[Special products]\label{le:rider}
  If $(X_n)_{n\geq1}$ are i.i.d.\ real random variables such that
  \begin{enumerate}
  \item[(i)] $\dE(e^{\theta |X_1|})<\infty$ for some $\theta>0$;
  \item[(ii)] $\dP(X_1\geq -a)=1$ for some $a>0$;
  \item[(iii)] $m=\dE(X_1)>0$ and $\si^2=\mathrm{Var}(X_1)>0$;
  \end{enumerate}
  then 
  \[
  \forall x\in\dR,\quad%
  \lim_{n\to\infty}%
  \prod_{k=1}^n \dP\PAR{\frac{X_1+\cdots+X_k}{n}\leq\frac{x}{A_n}+B_n}%
  =e^{-e^{-x}}
  \]
  where
  $c_n:=2\log\PAR{\frac{\sqrt{n/(2\pi)}}{\log(n)}}=\log(n)-2\log\log(n)-\log(2\pi)$
  and
  \[
  A_n:=\frac{\sqrt{nc_n}}{\si}
  \quad\text{and}\quad
  B_n:=m+\si\sqrt{\frac{c_n}{n}}.
  \]
\end{lemma} 

Note that in the case were $X_k\sim\Gamma(a_k,\la)$ for every $k$ then every
probability in the product can be expressed in terms of the incomplete Gamma
function.

\begin{proof} 
  We adapt the argument used by Rider \cite{MR1986426} in the case of
  exponential random variables. The driving intuitive idea is based on the Law
  of Large Numbers and the Central Limit Theorem, which indicate that the
  probabilities under the product are either asymptotically $1$ or identical
  to a Gaussian deviation probability, leading to the maximum of independent
  Gaussians, which is known to be Gumbel. Namely, let ${(K_n)}_{n\geq1}$ and
  ${(\de_n)}_{n\geq1}$ be deterministic sequences such that $K_n=\cO(\log(n))$
  and $\de_n=\cO(\sqrt{\log(n)}/\sqrt{n})$, and let $f_n$ be the increasing
  function in both $n$ and $x$ given by
  \[ 
  f_n(x):= %
  \sqrt{2\log\PAR{\frac{e^{x}\sqrt{n/(2\pi)}}{\log(n)}}}
  =\sqrt{2x+c_n}.
  \]
  The choices of $A_n$ and $B_n$ come from the following equivalence
  as $n\to\infty$
  \[
  \si\frac{f_n(x)}{\sqrt{n}}+m
  = \si\sqrt{\frac{2x+c_n}{n}}+m
  \sim \si\sqrt{\frac{c_n}{n}}\PAR{1+\frac{x}{c_n}}+m
  =:\frac{x}{A_n}+B_n.
  \]
  From Theorem \ref{th:layers} we have 
  \[
  \dP\PAR{\ABS{z}_{(1)}\leq\si\frac{f_n(x)}{\sqrt{n}}+m}
  =\prod_{k=1}^{n} %
  \dP\PAR{\frac{1}{n}\sum_{i=1}^{n-k}X_i\leq\si\frac{f_n(x)}{\sqrt{n}}+m}.
  \]
  We will show that the product over $1\leq k\leq\de_n n$ terms goes to
  $e^{-e^{-x}}$ as $n\to\infty$ while the product over $\de_n n\leq k\leq n-1$
  tends to $1$ as $n\to\infty$. Namely, denoting 
  \[
  Y_i:=\frac{X_i-m}{\si}
  \quad\text{and}\quad
  g_{n,k}(x):=\sqrt{\frac{n}{n-k}}\PAR{f_n(x)+\frac{m}{\si}\frac{k}{\sqrt{n}}},
  \]
  we get,
  \[
  \BRA{\frac{1}{n}\sum_{i=1}^{n-k}X_i\leq\si\frac{f_n(x)}{\sqrt{n}}+m} \\
  =\BRA{\frac{1}{\sqrt{n-k}} \sum_{i=1}^{n-k} Y_i\leq g_{n,k}(x)}.
  \]
  Let now $L$ be chosen large enough. We can assume that $n$ is large
  enough so that $K_n\gg\sup_{|x|\leq L} |f_n(x)|.$ Using quantitative
  Central Limit Theorem (Edgeworth expansion
  \cite{bhattacharya1976normal}) and assuming $|x|\leq L$, one has
  that for a given $k\in [n \delta_n, n]$
  \begin{multline*}
    \log \dP\PAR{\frac{1}{\sqrt{n-k}}\sum_{i=1}^{n-k}Y_i\leq g_{n,k}(x)}\\
    =\log\int_{-K_n}^{g_{n,k}(x)}\!\frac{e^{-t^2/2}}{\sqrt{2\pi}}\,dt%
    +\cO\PAR{\frac{1}{\sqrt{n}}\sup_{|x|\leq L}|f_n^2(x)|e^{-f_n^2(x)/2}+1/n+K_n/n^{3/2}}\\+\cO(e^{-cK_n^2}).
  \end{multline*}
  The last term stems from Hoeffding's concentration inequality to
  bound from above the probability that
  $\frac{1}{\sqrt{n-k}}\sum_{i=1}^{n-k}Y_i\leq -K_n$.  Indeed, first
  one may use (ii) to get
  \[
  \dP\PAR{\frac{1}{\sqrt{n-k}}\sum_{i=1}^{n-k}Y_i\leq-K_n} %
  \leq \dP\PAR{\frac{1}{\sqrt{n-k}}\sum_{i=1}^{n-k}Y_i\mathbf{1}_{|X_i-m|\leq m+a}\leq -K_n}.
  \]
  % En fait je n'arrive pas \`a montrer son truc en $exp(-\sqrt{n}K_n)$
  Next, when $k$ runs from $1$ to $n\de_n=\cO(\sqrt{n\log(n)})$ and when
  $n\gg1$ and $n\de_n\gg1$, one may see $g_{n,k}$ as an interpolation between
  $f_n(x)$ and $+\infty$. By copying the arguments of \cite{MR1986426}, which
  essentially amounts to the Riemann sum approximation of an integral, we
  obtain
  \begin{multline*}
    \sum_{1\leq k<n\de_n}\log\dP\PAR{\frac{1}{\sqrt{n-k}}\sum_{i=1}^{n-k}Y_i\leq g_{n,k}(x)}\\
    \quad\quad=\sqrt{n}\int_{f_n(x)}^{\infty}%
    \log\PAR{\int_{-\infty}^t e^{-s^2/2}\frac{ds}{\sqrt{2\pi}}}\,dt\\%
    +\cO\PAR{\frac{\log(n)}{\sqrt n} %
      +\sqrt{\log(n)}\sup_{|x|\leq L}f_n^2(x)e^{-f_n^2(x)/2}}.
  \end{multline*}
  With our choice of $f_n(x)$, we get \cite[eq.\ (15)]{MR1986426}
  that uniformly on compact sets in $x$
  \[
  \lim_{n \to \infty}\log\prod_{1\leq k<n\de_n} %
  \dP\PAR{\frac{1}{n}\sum_{i=1}^{n-k}X_i\leq\si\frac{f_n(x)}{\sqrt{n}}+m} %
  =-e^{-x}.
  \]
  We now examine the contribution of the remaining terms:
  \[
  \prod_{k=n\de_n}^{n-1}
  \dP\PAR{\frac{1}{n}\sum_{i=1}^{n-k}X_i\leq\si\frac{f_n(x)}{\sqrt{n}}+m}
  \]
  As the product is never larger than $1$, one only needs to get a lower
  bound. Now, by (i),
  \begin{align*}
    \prod_{k=n\de_n}^{n-1}
    \dP\PAR{\frac{1}{n}\sum_{i=1}^{n-k}X_i\leq\si\frac{f_n(x)}{\sqrt{n}}+m}
     &\geq \prod_{k=n \delta_n}^{n-1}\PAR{1-\dP\PAR{\sum_{i=1}^{n-k}
        X_i\geq\si\sqrt{n}f_n(x)+nm}}\\
     &\geq \prod_{k=n\de_n}^{n-1}\PAR{1-e^{-\te\si\sqrt{n}f_n(x)-\te nm}\dE(e^{\te X_1})^{n-k}}
  \end{align*}
  For any $k$ such that $(n-k)/n<c<1$ it is not difficult to see that
  \[
  1-e^{-\te\si\sqrt{n}f_n(x)-\te nm}\dE(e^{\te X_1})^{n-k}\geq (1-e^{-C n}),
  \]
  where the constant $C>0$ depends on $c<1.$ All the difficulty lies in the
  regime where $n-k \sim n.$ In this case we can use the moderate deviation
  result of
  % V. Petrov , A generalization of Cramers limit theorem. Select. Transl.
  % Math.
  % Statist. Probab. 6 (1966), 1--8, American Mathematical Society,
  % Providence,
  % R.I. THIS IS REFERENCE [327] IN PETROV'S BOOK
  \cite[Th.~5.23~p.~189]{MR1353441} to get that for any $n \delta_n \leq k
  \leq n(1-\epsilon)$
  \[
  \dP\PAR{\sum_{i=1}^{n-k} X_i\geq\si\sqrt{n}f_n(x)+nm}\leq e^{- C_0 \log(n-k)}.
  \]
  Actually we can refine the estimate when $n^{3/4}\leq k \leq (1-\epsilon)n$
  \[
  \dP\PAR{\sum_{i=1}^{n-k} X_i\geq\si\sqrt{n}f_n(x)+nm}\leq e^{-C_0 n^{1/4}}.
  \]
  Combining the whole, we deduce that 
  \begin{multline*}
    \prod_{k=n \delta_n}^{n-1}\dP\PAR{\frac{1}{n}\sum_{i=1}^{n-k}X_i\leq\si\frac{f_n(x)}{\sqrt{n}}+m}\\
    \geq (1-e^{-cn})^{n(1-1/c)}(1- e^{- C_0 n^{1/4}})^{nc}(1-e^{-C_0\log(n-k)})^{n \delta_n},
  \end{multline*}
  which obviously goes to $1$ as $n \to \infty.$
  
  % \textbf{Il est possible que j'ai rate qqchose car je ne vois pas comment
  % son
  % estimee suffit au milieu de la page 3405 pour dire que c'est
  % essentiellement
  % 1 pour en deduire ensuite 7).}
\end{proof}

\begin{lemma}[An instance of the delta method]\label{le:delta}
  Let $(Z_n)$ be a sequence of random variables on $(0,\infty)$ and let
  $f:(0,\infty)\to\dR$ be a $\cC^1$ function such as $f(x)=x^\rho$ with
  $\rho>0$. If there exist deterministic sequences $(A_n)$ and $(B_n)$ in
  $(0,\infty)$ such that
  \begin{enumerate}
  \item[(i)] $A_n(Z_n-B_n)$ converges in law as $n\to\infty$ to some
    probability distribution $P$;
  \item[(ii)] $A_n\to\infty$;
  \item[(iii)] $B_n\to B>0$
  \item[(iv)] $f'(B)\neq0$;
  \end{enumerate}
  then, denoting $a_n=A_n/f'(B)$ and $b_n=f(B_n)$, the sequence
  $a_n(f(Z_n)-b_n)$ converges also in law as $n\to\infty$ to the same
  probability distribution $P$.
\end{lemma}

\begin{proof}
  We may use a Taylor formula (or intermediate value theorem) to get
  \[
  f(Z_n)-f(B_n)=f'(W_n)(Z_n-B_n)
  \]
  where $W_n$ is a random variable lying between $B_n$ and $Z_n$. Thanks to
  (i), (ii), and the Slutsky lemma, we have $Z_n-B_n\to0$ in law, and thus in
  probability because the limit is deterministic. Therefore $W_n-B_n\to0$ in
  probability, and thus, by (iii), $W_n\to B$ in probability. Since $f'$ is
  continuous at point $B>0$, the continuous mapping theorem gives that
  $f'(W_n)\to f'(B)$ in probability, and thus, by (iv), $f'(W_n)/f'(B)\to1$ in
  probability. Now it remains to use (i) and the Slutsky lemma to obtain 
  \[
  a_n(f(Z_n)-b_n)
  =\frac{A_n}{f'(B)}(f(Z_n)-f(B_n))
  =\frac{f'(W_n)}{f'(B)}A_n(Z_n-B_n)
  \overset{d}{\longrightarrow}P.
  \]
\end{proof}

\section{Proof of Theorem \ref{th:genecase}}
\label{se:genecase}

This section is devoted to the proof of Theorem \ref{th:genecase}. We thus
now consider the case where the potential $V$ is rather general satisfying the
following assumptions:
\begin{itemize}
\item \textbf{(A1)} $V$ is strictly convex: there exists $a>0$ such that
  $V''(u)\geq a, \forall u\geq 0.$
\item \textbf{(A2)} For each $x\in [0,2]$, there exists a unique $t_x$ such
  that $t_x V'(t_x)=2-x.$
\end{itemize}
Theorem \ref{th:genecase} follows from the Lemma \ref{le:genecase}
below. Namely, Lemma \ref{le:genecase} gives
\[
\forall u\in\dR,\quad
\lim_{n\to\infty}\dP\PAR{\sqrt{n}\PAR{\ABS{z}_{(1)}-t_0}\leq f_n(u)} = e^{-e^{-u}}
\]
where
\[
f_n(u)
=\sqrt{2u+c_n}
=\sqrt{c_n}\sqrt{1+2u/c_n}
=\sqrt{c_n}(1+u/c_n+o(c_n)),
\]
which leads to 
\[
\forall
u\in\dR,\quad\lim_{n\to\infty}\dP\PAR{a_n\PAR{\ABS{z}_{(1)}-b_n}\leq u}=e^{-e^{-u}}
\]
where
\[
a_n:=\frac{\sqrt{nc_n}}{C_0}
\quad\text{and}\quad
b_n:=t_0+C_0\sqrt{\frac{c_n}{n}}.
\]

\begin{lemma}\label{le:genecase}
  Let $t_0$ be the unique solution to the equation : $t_0 V'(t_0)=2.$ Set also
  \[
  C_0 := \frac{1}{\sqrt{|F''(t_0)|^{3/2}\frac{1}{2}t_0}} %
  \quad\text{and}\quad %
  f_n(u) := \sqrt{2\log\PAR{\frac{e^{u}\sqrt{n/(2\pi)}}{\log(n)}}}
  \]
  where $F''(t_0):=-2/t_0^2-V''(t_0)$. Then for every $U_0>0$, uniformly in $u
  \in [-U_0, U_0]$,
  \begin{equation*}
    \lim_{n \to \infty}\dP\PAR{\ABS{z}_{(1)}\leq t_0+C_0\frac{f_n(u)}{\sqrt n}}
    =e^{-e^{-u}}. 
  \end{equation*}
\end{lemma}

\begin{remark}[Relaxed assumption] Assumption \textbf{(A1)} can in
  principle be relaxed to the strict convexity of
  $V(t)-\frac{2k-1}{n}\log(t)$ for all $1\leq k \leq n.$ The
  asymptotic expansion can then be performed unchanged provided that
  there exists $a>0$ so that for any such $k$
  \[
  \forall t>0,\quad V''(t)+\frac{2k-1}{n}t^{-2}>a.
  \]
\end{remark}

\begin{proof}[Proof of Lemma \ref{le:genecase}]
  Recall that $(\ABS{z}_{(1)},\ldots,\ABS{z}_{(n)})$ is distributed as
  the order statistic $(X_{(1)},\ldots,X_{(n)})$ of $X_1,\ldots,X_n$
  which are independent random variables with $X_i$ of density
  proportional to $t\mapsto t^{2i-1}e^{-nV(t)}.$ We shall first
  consider the value (if it exists) around which the random variables
  $X_i$ concentrate for a large $i.$ Typically if $\lim_{n \to
    \infty}(i/n)=2-x$, we here show that $X_i$ concentrates around the
  deterministic number $t_x.$ Indeed, if $f$ be a bounded continuous
  function $f:\R\to \R$, then we consider the asymptotic expansion of
  $\dE(f(X_i))$ when $i/n\to 2-x>0$ using the Laplace method.  We have
  \[
  \dE(f(X_i))=\frac{1}{Z_i}\int_0^\infty\!e^{nF_i(t)}f(t)\,dt,
  \]
  where 
  \[
  Z_i:=\int_{0}^{\infty}\!t^{2i-1}e^{-nV(t)}dt 
  \quad\text{and}\quad
  F_i(t):=\frac{2i-1}{n}\log(t)-V(t).
  \]
  To determine critical points, one uses that
  \[
  F_i'(t)=\frac{2i-1}{nt}-V'(t).
  \]
  By assumption \textbf{(A2)}, there exists a unique solution $t_{(i)}$ to the
  equation $F_i'(t)=0$ and by \textbf{(A1)} one has that $|t_{(i)}-t_{x}| \leq
  C_0(1/n+ |x-i/n|)$ for some constant $C_0$ (independent of $x,i,n$). Also by
  \textbf{(A1)} this critical point is non degenerate as
  \[
  F_i''(t)=-\frac{2i-1}{nt^2}-V''(t)<-a.
  \]
  Note that this readily implies the Gaussian decay of the exponential term:
  for all $t\in\R$,
  \[
  F_i(t)-F_i(t_i)\leq -\frac{a}{2}(t-t_i)^2.
  \]
  By a standard Laplace approximation, one then deduces that for any $x\in
  [0,2[$ and for any $i$ such that $i/n \to 2-x$ the following holds:
  \[
  \dE f(X_i)=f(t_i) (1+\cO(n^{-1/2}))=f(t_x)(1+o(1)).
  \]
  As a consequence, if $i$ is very large, that is those integers $i$ such that 
  \[
  \frac{i}{n}=1-\epsilon_n, \quad \text{with }\epsilon_n \to 0 \text{ as }n \to \infty,
  \]
  using that 
  \[
  t_i-t_0=-\frac{-2}{|t_0 F''(t_0)|}\PAR{\epsilon_n+\frac{1}{n}} +o(\epsilon_n)
  \]
  we conclude that the $X_i$'s concentrate around $t_0.$ The constant
  $t_0$ will play the role of $a_n$ in the previous analysis.
  
  Let us now try to give the main ideas of the rest of the proof. Let
  $u>0$ be a given real number. Using Theorem \ref{th:layers}, in
  order to determine the limiting distribution of $\ABS{z}_{(1)}$ we
  consider the following:
  \[
  \dP\PAR{\ABS{z}_{(1)}\leq t_0+\frac{g_n(u)}{\sqrt{n}}} %
  =\prod_{i=1}^n\dP\PAR{X_i \leq t_0+\frac{g_n(u)}{\sqrt{n}}},
  \]
  where $g_n(x)$ is growing to infinity with $n$ and to be
  determined. Consider first
  \[
  \prod_{i=i_0}^n\dP\PAR{X_i \leq t_0+\frac{g_n(u)}{\sqrt{n}}},
  \]
  for some $i_0$ that we choose as
  \[
  i_0:=n- c\sqrt{n\log(n)}.
  \]
  The value of the real $c>0$ will be fixed later.  One can first
  observe that $i_0\geq \epsilon n$ for some $\epsilon>0$. For each
  $i\geq i_0$ one has that
  \begin{align}
    \dP\PAR{X_i \leq t_0+\frac{g_n(u)}{\sqrt n}}
    &=\frac{1}{Z_i}\int_{-\infty}^{t_0+n^{-1/2}g_n(u)}\!e^{n F_i(t)}\,dt\nonumber\\
    &=\frac{(F_i''(t_i) n)^{1/2}}{\sqrt{2\pi }e^{nF_i(t_i)}(1+\cO(n^{-3/4}))}\int_{-\infty}^{t_0+n^{-1/2}g_n(u)}\!e^{n F_i(t)}\,dt\nonumber\\
    &=\frac{(F_i''(t_i) n)^{1/2}}{\sqrt{2\pi } (1+\cO(n^{-3/4}))}\int_{-\infty}^{t_i +t_0-t_i+n^{-1/2}g_n(u)}\!e^{n( F_i(t)-F_i(t_i))}\,dt\nonumber\\
    &=\frac{ (F_i''(t_i) )^{1/2}}{\sqrt{2\pi
      }(1+\cO(n^{-3/4}))}\int_{-\infty}^{\sqrt{n}(t_0-t_i)+g_n(u)}\!e^{n(
      F_i(t_i+\frac{s}{n^{1/2}})-F_i(t_i))}\,ds. 
    \label{eq:intfinal}
  \end{align}
  Hereabove we have used the asymptotic expansion of $Z_i$ using the Laplace
  method, and the fact that $i\geq \epsilon n$ for the error control, and in the
  last line we made the change of variables $t=t_i+s/\sqrt n$.
  
  Assume now that $g_n(u)\ll n^{1/2}.$ Because $t_i\leq t_0$, the
  integral in \eqref{eq:intfinal} is of order $\cO(1)$. We need to
  refine this rough order.  To that aim, we need to control the
  variation of $F_i$ (and the Gaussian decay) in a bounded
  neighborhood of $t_i$. Let $\delta>0$ be given (small), that we fix
  hereafter. Let us first examine the decay of $F_i$ ``far'' from
  $t_i$. Using the upper bound $F_i(t)-F_i(t_i)\leq -a(t-t_i)^2/2$
  forall $t$, one can see that only a neighborhood of $t_i$ of width
  $n^{-1/2}$ can contribute to the integral. We then turn to a
  neighborhood of $t_i$.
\[
\ABS{nF_i(t_i+s)-nF_i(t_i)-nF_i''(t_i) \frac{s^2}{2}} %
\leq n \sup_{|t-t_i|\leq \delta} |F_I'''(t)| \frac{s^3}{3!} %
\leq n \delta \sup_{|t-t_i|\leq \delta} |F_I'''(t)| \frac{s^2}{6}
\]
One can choose $\delta$ small enough so that for any $s, |s|<\delta$
\[
\ABS{nF_i(t_i+s)-nF_i(t_i)-nF_i''(t_i) \frac{s^2}{2}}\leq n F_{i}''(t_i)\frac{s^2}{4}.
\]
Thus using the change of variables $t:=t_i+s/n^{1/2}$ and by a straightforward
Taylor expansion, we deduce that for all $i\geq i_0$,
\begin{multline}
  \Big | \int_{|s|\leq \delta \sqrt n}e^{n(
    F_i(t_i+\frac{s}{n^{1/2}})-F_i(t_i))}ds
  -\int_{|s|\leq \delta \sqrt n}e^{-F_i''(t_i)^2s^2/2}ds\Big |\\
  \leq \sup_{i\geq i_0}\int_{|s|\leq \delta \sqrt n} n^{-1/2}\sup_{z\in B(t_0,
    \delta)}|F_{i}'''(z)| \frac{s^3}{6} e^{-C_0s^2} \leq
  C_1n^{-1/2}.\label{eq:esterrorini}
\end{multline}
The above estimate on the error $\cO(n^{-1/2})$ is not precise enough as,
later, one will sum this error over $n^{1/2}\sqrt{\log(n)} $ integers $i\geq
i_0$. We need to refine the above estimate.

\begin{lemma}[Exponential expansion]\label{le:base} 
  If $h: \R \to \R$ is $\cC^4$ function such that $h'(0)=0$ and $h''(0)<0$
  then there exists $\de>0$ such that for all $|t|\leq \delta$,
  \[
  e^{h(t)-h(0)}-e^{\frac{h''(0)t^2}{2}} %
  =\frac{h^{(3)}(0)}{6}t^3 e^{\frac{h''(0)t^2}{2}}+\cO(P(t))e^{\frac{h''(0)t^2}{4}},
  \]
  where $P(t)=C_1t^4+C_2t^6$ for some constants $C_1>0$ and $C_2>0$.
\end{lemma}

\begin{proof}[Proof of Lemma \ref{le:base}]
  First, one can write
  \[
  h(t)-h(0) %
  =\int_0^t\!\int_0^s\!h''(u)\,duds %
  =t^2\int_0^1\!s\,ds\int_0^1\!h''(sut)\,du.
  \]
  From this, one deduces that 
  \[
  e^{h(t)-h(0)}-e^{\frac{h''(0)t^2}{2}} %
  =e^{\frac{h''(0)t^2}{2}} %
  \PAR{e^{t^2\int_0^1\!s\,ds\int_0^1\!(h''(sut)-h''(0))\,du}-1}.
  \]
  Set now $G(t):=t^2\int_0^1sds\int_0^1 du (h''(sut)-h''(0)).$
  Then again 
  \[
  G(t)=t^3\int_0^1 s^2ds\int_0^1 udu \int_0^{1} h^{(3)}(sutx) dx.
  \]
  Using that 
  \[
  e^X-1=X+X^2\int_0^1 t dt \int_0^1 du e^{Xtu},
  \]
  one then deduces that  
  \begin{align*}
    \frac{e^{h(t)-h(0)}-e^{\frac{h''(0)t^2}{2}}}{e^{\frac{h''(0)t^2}{2}}}
    &=e^{G(t)}-1\\
    &=G(t)+G(t)^2\int_0^1\!vdv\int_0^1\!due^{G(t)vu}x\\
    &=\frac{t^3}{6}h^{(3)}(0)+t^4\int_0^1\!s^3ds\int_0^1\!u^2du
    \int_0^{1}xdx\int_0^1\!h^{(4)}(sutxv)\,dv\\
    &\quad+G(t)^2\int_0^1\!vdv \int_0^1\!due^{G(t) vu}.
  \end{align*}
  Thus we get that, every $\de>0$ and for all $t$ such that $|t|\leq \delta,$
  \[
  \ABS{e^{h(t)-h(0)}-e^{\frac{h''(0)t^2}{2}}-e^{\frac{h''(0)t^2}{2}}\frac{t^3}{6}
    h^{(3)}(0)}%
  \leq e^{\frac{h''(0)t^2}{2}}\PAR{t^4C_1\sup_{|t|\leq \delta}|h^{(4)}(t)| %
    +C_2 t^6 e^{G(t)\vee 0}}.
  \]
  As $|G(t)|\leq C_3t^3$ one then deduces that for $\delta$ small enough and
  any $|t|\leq \delta$,
  \[
  e^{\frac{h''(0)t^2}{2}+G(t)}\leq e^{\frac{h''(0)t^2}{4}}.
  \]
  This yields the desired estimate.
\end{proof}

Back to the proof of Lemma \ref{le:genecase}, thanks to Lemma
\ref{le:base}, \eqref{eq:esterrorini} can be improved to
\begin{multline*}
  \Big | \int_{|s|\leq \delta \sqrt n}e^{n(
    F_i(t_i+\frac{s}{n^{1/2}})-F_i(t_i))}ds %
  -\int_{|s|\leq \delta \sqrt n}e^{-F_i''(t_i)^2\frac{s^2}{2}}ds\\%
  -\int_{|s|\leq \delta \sqrt
    n}e^{-F_i''(t_i)^2\frac{s^2}{2}}F_i^{(3)}(t_i)\frac{s^3}{6n^{1/2}}ds\Big |
  \\\leq \int_{|s|\leq \delta \sqrt n}
  \frac{C_1s^4}{n}+\frac{C_2s^6}{n^2}e^{-C_0s^2} \leq C_3n^{-1}.
  \end{multline*}
  Note first that the above estimates hold uniformy in $i\geq i_0$
  provided that $i_0/n \to 1$ as $n$ goes to infinity. In particular
  the constants $C_1, C_2, C_3$ do not depend on $i\geq i_0.$ Note
  also that the same estimate holds if one multiplies the integrand by
  the indicator function $\mathbf{1}_{s\leq \sqrt n
    (t_0-t_i)+g_n(u)}.$ Let now $h_n <\de\sqrt{n}$ be such that
  $h_n\to\infty$ as $n$ grows to infinity. Then, using the fact that
  the third derivative $F_i'''(t_i)$ is bounded uniformly in $i$ when
  $n-i\leq c\sqrt{n\log(n)}$, we deduce that
  \[
    \int_{|t-t_i|\leq\de, s\leq h_n}\!
    e^{-F_i''(t_i)^2s^2/2}F_i^{(3)}(t_i)\frac{s^3}{6n^{1/2}}\,ds
    \leq Cn^{-1/2} h_n^2 e^{-h_n^2}.
  \]
  Indeed the integral can be explicitly computed. Lastly one has that there
  exists a constant $C_4$ such that for any $i\geq i_0$
  \[
  \ABS{\int_{t<t_i-\delta}\!e^{n( F_i(t_i+\frac{s}{n^{1/2}})-F_i(t_i))}\,ds}%
  \leq C_4e^{-na \delta^2/2} %
  \quad\text{and}\quad %
  \ABS{\int_{t<t_i-\delta}\!e^{nF_i''(t_i)s^2/2}\,ds}\leq
  C_4e^{-na\delta^2/2}.
  \]
  Combining the whole yields that up to an error term (of order
  $n^{-1/2}h_n^2e^{-h_n^2}$) one can replace the exponential term by
  the Gaussian one obtained by the Taylor expansion.

  Let $L$ be some large positive real number. Here we assume that
  $|u|\leq L$ and prove uniform convergence on compact sets of the
  c.d.f.: one has that
  \begin{multline}
    \log\prod_{i=i_0}^n\dP\PAR{X_i \leq t_0+\frac{g_n(u)}{\sqrt n}}\\
    =\sum_{i\geq
      i_0}\BRA{\log\SBRA{\int_{-\infty}^{g_n(u)+\sqrt{n}(t_0-t_i)}\frac{\sqrt{F_i''(t_i)}}{\sqrt{2\pi}}e^{-F''_i(t_i)
          s^2/2} ds+ \cO(n^{-1/2} h_n^2e^{-h_n^2/2})}}, \label{eq:majoerreur}
  \end{multline}
  where $h_n=\sup_{|u|\leq L}g_n(u).$ Then the sum over $i$ of the
  error term in \eqref{eq:majoerreur} can be bounded from above by
  $c\sqrt{\log(n)}h_n^2e^{-h_n^2/2}$.  This error term is negligible
  provided $h_n^2\gg 2{\log\log(n)}.$
  % Note first that because $t_i-t_0=\cO(1/n)$ we can replace in the
  % Gaussian
  % $F_i''(t_i)$ with $F_0''(t_0)$. Also if we understand the above
  % sum as a
  % Riemann integral as in RIDER then we essentially have his formula
  % (10)...
  To sum up, we have
  \begin{align*}
    &\log \prod_{i=i_0}^n\dP\PAR{X_i \leq t_0+\frac{g_n(u)}{\sqrt n}}\\
    &=\sum_{i\geq i_0}\log
    \int_{-\infty}^{\sqrt{|F_i''(t_i)|}(g_n(u)+\sqrt{n}(t_0-t_i))}
    \frac{1}{\sqrt{2\pi}}e^{- s^2/2} ds+ %
    \cO((\log(n))^{1/2}h_n^2e^{-h_n^2/2})\\
    &=\sum_{i\geq
      i_0}\log\int_{-\infty}^{\sqrt{|F''(t_0)|}(g_n(u)+\sqrt{n}(t_0-t_i))+g_n(u)(\sqrt{|F_i''(t_i)|}-\sqrt{|F''(t_0)|})
    }\frac{1}{\sqrt{2\pi}}e^{-s^2/2} ds\\
    &\quad +\cO((\log(n))^{1/2}h_n^2e^{-h_n^2/2})
  \end{align*}
  where we have noted $F''(t_0):=-2/t_0^2-V''(t_0).$ Let us now choose
  \[
  g_n(u):=\frac{f_n(u)}{\sqrt{|F''(t_0)|^{3/2}t_0/2}} %
  \quad\text{where}\quad %
  f_n(u)= \sqrt{2\log\PAR{\frac{e^{u}\sqrt{n/(2\pi)}}{\log(n)}}}.
  \]
  Then $f_n\to \infty$ provided $u$ is bounded from below and one can
  easily check that with such a choice of $f$ all the required
  estimates on $h_n$ hold true. Also there exists a constant $C(t_0)$
  such that for any $i\geq i_0$
  \[
  F_i''(t_i)-F''(t_0)= \frac{n-i}{n} (C_0+o(1)).
  \]
  The constant $C_0$ can be expressed in terms of $t_0$ and the
  derivatives of $V$: it may be zero. The important fact is that
  \[
  \sqrt{|F''(t_0)|}\sqrt{n}(t_0-t_i))+g_n(u)(\sqrt{|F_i''(t_i)|}-\sqrt{|F''(t_0)|})%
  =\frac{n-i}{\sqrt{n}}\PAR{\frac{2}{t_0\sqrt{|F''(t_0)|}}+o(1)}.
  \]
  Using now the same arguments as in \cite{MR1986426} for the Riemann
  sum approximation, we again deduce that uniformly on compact
  subsets,
  \[
  \log\prod_{i=i_0}^n\dP\PAR{X_i\leq t_0+\frac{g_n(u)}{\sqrt n}}=-e^{-u}.
  \]
  Now there remains to consider the smaller integers
  $i$. %: first those $i\geq \sqrt n $ and then the remaining ones....
  Assume now that $i<i_0$. The exact critical point $t_i$ depends on
  $n$ and might now be closer and closer to $t_2$ where
  $V'(t_2)=0$. We need to control the second derivative of $F_i$ to be
  able to undertake the same Laplace analysis as above.
  \[
  F_i''(t)=-V''(t)-\frac{2i-1}{nt^2}.
  \]
  By the convexity assumption, we deduce that $t_i>t_2$ so that the
  second derivative is negative and bounded from above in the whole
  interval $[t_2/2, 2t_0].$ From that one can deduce that
  \[
  n \left ( F_i(t)-F_i(t_i) \right ) %
  =n \int_{t_i}^t\!F_i'(s)\,ds %
  =n \int_{t_i}^t\!\int_{t_i}^s\!F_i''(u)\,duds %
  \leq -n a (t-t_i)^2,
  \]
  by the strict convexity of $V$. In particular for all $i\leq i_0$
  one has that $t_i\leq t(i_0)$ with
  \[
  |t(i_0)-t_0| \geq C_5 \sqrt{\frac{\log(n)}{n}},
  \]
  for some constant $C_5>0$. Thus we deduce that
  \begin{align*}
    \prod_{i=0}^{i_0}\dP\PAR{X_i \leq t_0+\frac{g_n(u)}{\sqrt n}}
    &=\prod_{i=0}^{i_0}\PAR{1-\dP\PAR{X_i \geq t_0+\frac{g_n(u)}{\sqrt(n)}}}\\
    &=\prod_{i=0}^{i_0} \SBRA{1- \frac{\int_{t_0+\frac{g_n(u)}{\sqrt n}}^{\infty}e^{n F_i(t)}dt }{\int_{\R^+}e^{n F_i(t)}dt}}\\
    &\geq \prod_{i=0}^{i_0} \SBRA{1- (\sqrt{n}(1/2+o(1))\int_{t_0+\frac{g_n(u)}{\sqrt n}}^{\infty}e^{n F_i(t)-n F_i(t_i)}dt}\\
    &\geq \prod_{i=0}^{i_0} \SBRA{1-(\sqrt{n}(1/2+o(1))e^{-n a (t_0+\frac{g_n(u)}{\sqrt n}-t_i)^2}}\\
    & \geq \prod_{i=0}^{i_0} \SBRA{1-(\sqrt{n}(1/2+o(1))e^{-aC_5c^2\log(n)}}.
  \end{align*}
  Let us now choose $c$ (determining by this way $i_0$) so that $aC_5 c^2>2.$
  The latter product goes to $1$ as $n $ goes to infinity. This is the needed
  estimate.
\end{proof}

 %
 % Note: the asymptotic fluctuations of the extremes of independent and
 % non identically distributed sequences of random variables is studied in
 % \cite{MR882244} (see also \cite[Section 5.6 p. 201]{MR2234156}). Maybe not
 % so useful but worth of being noticed here.
 % 
 % quadratized ensembles of Khoruzhenko and their single ring. 
 %
 % \cite{MR2817648,ECP1818}
 %

%\bibliographystyle{amsalpha}
%%\bibliographystyle{amsplain}
%%\bibliography{cougas}

\providecommand{\bysame}{\leavevmode\hbox to3em{\hrulefill}\thinspace}
\providecommand{\MR}{\relax\ifhmode\unskip\space\fi MR }
\providecommand{\MRhref}[2]{%
  \href{http://www.ams.org/mathscinet-getitem?mr=#1}{#2}
}
\providecommand{\href}[2]{#2}

\end{document}